\documentclass[10pt]{article}
\font\smallit=cmti10

\usepackage{amssymb,latexsym,amsmath,epsfig,amsthm} 

\makeatletter

\renewcommand\section{\@startsection {section}{1}{\z@}
{-30pt \@plus -1ex \@minus -.2ex}
{2.3ex \@plus.2ex}
{\normalfont\normalsize\bfseries\boldmath}}

\renewcommand\subsection{\@startsection{subsection}{2}{\z@}
{-3.25ex\@plus -1ex \@minus -.2ex}
{1.5ex \@plus .2ex}
{\normalfont\normalsize\bfseries\boldmath}}

\renewcommand{\@seccntformat}[1]{\csname the#1\endcsname. }

\makeatother

\theoremstyle{definition}

\usepackage{amsfonts}
\usepackage{amssymb}
\usepackage{latexsym}
\usepackage{subfigure}
\usepackage{amsmath}
\usepackage{amsthm}
\usepackage{xcolor}
\usepackage{lineno}

\newtheorem{thm}{Theorem}[section]
\newtheorem{exa}[thm]{Example}
\newtheorem{lem}[thm]{Lemma}
\newtheorem{cor}[thm]{Corollary}
\newtheorem{prop}[thm]{Proposition}
\newtheorem{conj}[thm]{Conjecture}
\newtheorem{prob}[thm]{Problem}

\newtheorem{rem}[thm]{Remark}

\newcommand{\Z}{\mathbb{Z}}

\newcommand{\abegroup}{\mathbb{Z}_{n_1}\oplus \cdots \oplus \mathbb{Z}_{n_r}} 
\newcommand{\pgroup}[1]{\mathbb{Z}_{p^{\alpha_1}} \oplus  \cdots \oplus \mathbb{Z}_{p^{\alpha_{#1}}}} 

\newcommand{\seq}[1]{({#1})}

\newcommand{\s}[2]{S({#1},{#2})} 

\newcommand{\Dav}[1]{\mathsf{D}(#1)} 
\newcommand{\D}[3]{\mathsf{D}({#3},{#2})} 
\newcommand{\EGZ}[3]{\mathsf{EGZ}({#1}, {#2}, {#3})} 

\newcommand{\GG}{C}

\newcommand{\elem}[2]{e_{#1}({#2})} 
\newcommand{\Elem}[2]{e_{#1}(x_1, \ldots , x_{#2})} 
\newcommand{\Pow}[2]{p_{#1}(x_1, \ldots, x_{#2})}  
\newcommand{\num}[2]{v_{#1}({#2})} 

\newcommand{\Low}[2]{L({#1},{#2})} 

\newcommand{\dom}[1]{t({#1})} 

\newcommand{\zz}{\mathbf{z}}

\makeatletter
\def\imod#1{\allowbreak\mkern10mu({\operator@font mod}\,\,#1)}
\makeatother

\begin{document}

\begin{center}
\uppercase{\bf Higher Degree Erd\H{o}s-Ginzburg-Ziv Constants}
\vskip 20pt
{\bf Yair Caro}\\
{\smallit Department of Mathematics, University of Haifa-Oranim, Israel}\\
{\tt yacaro@kvgeva.org.il}\\ 
\vskip 10pt
{\bf John R. Schmitt}\\
{\smallit Department of Mathematics, Middlebury College, Middlebury, Vermont, USA}\\
{\tt jschmitt@middlebury.edu}

\end{center}
\vskip 20pt

\vskip 30pt

\centerline{\bf Abstract}

\noindent
We generalize the notion of Erd\H{o}s-Ginzburg-Ziv constants -- along the same lines we generalized in earlier work the notion of Davenport constants --  to a ``higher degree" and obtain various lower and upper bounds.  These bounds are sometimes exact as is the case for certain finite commutative rings of prime power cardinality.  We also consider to what extent a theorem due independently to W.D.~Gao and the first author that relates these two parameters extends to this higher degree setting.  Two simple examples that capture the essence of these higher degree Erd\H{o}s-Ginzburg-Ziv constants are the following. 1) Let $\nu_p(m)$ denote the $p-$adic valuation of the integer $m$.  Suppose we have integers $t | {m \choose 2}$ and $n=t+2^{\nu_2(m)}$, then every sequence $S$ over $\Z_2$ of length $|S| \geq n$ contains a subsequence $S'$ of length $t$ for which  $\sum_{a_{i_1},\ldots, a_{i_m} \in S'} a_{i_1}\cdots a_{i_m} \equiv 0 \pmod{2}$, and this is sharp. 2) Suppose $k=3^{\alpha}$ for some integer $\alpha \geq 2$.  Then every sequence $S$ over $\Z_3$ of length $|S| \geq k+6$ contains a subsequence $S'$ of length $k$ for which $\sum_{a_h, a_i, a_j \in S'} a_ha_ia_j \equiv 0 \pmod{3}$.  These examples illustrate that if a sequence of elements from a finite commutative ring is long enough, certain symmetric expressions (symmetric polynomials) have to vanish on the elements of a subsequence of prescribed length.  The Erd\H{o}s-Ginzburg-Ziv Theorem is just the case where a sequence of length $2n – 1$ over $\Z_n$ contains a subsequence $S'=(a_1, \ldots, a_n)$ of length $n$ that vanishes when substituted in the linear symmetric polynomial $a_1+\cdots+a_n.$

\pagestyle{myheadings}
\thispagestyle{empty}
\baselineskip=12.875pt
\vskip 30pt

\section{Introduction}\label{sec:introduction}

Throughout this paper, let $p$ denote a prime number and $q=p^{\alpha}$ a prime power.

Let $G$ be a finite abelian group with $exp(G)$ its exponent.  Then for $g_i \in G$ $$S=\seq{g_1, \ldots, g_{\ell}}=\prod_{g \in G}g^{\num{g}{S}}$$  is called a {\it sequence over $G$}, where order is disregarded, repetition is allowed and the exponent $\num{g}{S}$ indicates the number of repetitions of the element $g$ in $S$.  Its {\it length}, denoted $|S|$, is the number of elements counted with multiplicity, i.e. $|S|=\sum_{g \in G}v_g(S)$.  A sequence of $G$ is said to be {\it zero-sum} if the sum of its elements is zero in $G$.  A sequence $S$ of $G$ is said to be {\it zero-sum free} if every non-trivial subsequence of $S$ has sum different to zero.  

For a group $G$, the {\it Davenport constant of $G$}, which we denote by $\Dav{G}$, is the smallest positive integer $z$ such that every sequence $S$ over $G$ of length $|S| \geq z$ contains a non-empty zero-sum subsequence, that is, $S$ is not zero-sum free.  For a group $G$, the {\it Erd\H{o}s-Ginzburg-Ziv constant of $G$} is the smallest positive integer $z$ such that every sequence of length $|S|\geq z$ contains a zero-sum subsequence of length $|G|$.

These two constants have been well-studied; see, for instance, the survey paper of W.D.~Gao and A.~Geroldinger \cite{GaoGeroldinger}.  We recall some of the earlier statements as follows.  

Recall that by the Fundamental Theorem of Finite Abelian Groups that for any finite non-trivial abelian group $G$ there exist integers $n_1, \ldots, n_r$ with $1 <n_1|\ldots|n_r$ so that $G$ can be written uniquely as

$$G \cong \abegroup,$$

\noindent where the integer $r$ is called the {\em rank} of $G$ and denoted $r(G)$.  We use $\mathsf{d}^*(G)$ to denote the value $\sum_{i=1}^r(n_i-1)$.

The value of $\Dav{G}$ was determined independently by J.E.~Olson \cite{Olson69a} and D. Kruyswijk \cite{vanEmdeBoasKruyswijk67} when $G$ is a $p$-group, and by J.E.~Olson \cite{Olson69b} when $G$ has rank at most $2$.

\begin{thm}[J.E.~Olson \cite{Olson69a}, \cite{Olson69b}, and D.~Kruyswijk \cite{vanEmdeBoasKruyswijk67}]\label{thm:Olson}
If $G$ is a $p$-group or $r(G) \leq 2$, then $\Dav{G} = 1+\mathsf{d}^*(G)$.
\end{thm}

P.~Erd\H{o}s, A.~Ginzburg, A.~Ziv \cite{ErdosGinzburgZiv61} showed that for the cyclic group $\Z_k$ the smallest positive integer $z$ such that every sequence of length $|S|\geq z$ contains a zero-sum subsequence of length $|\Z_k|=k$ is $2k-1$.

One particular exciting result that connects these two constants is due independently to W.D.~Gao \cite{Gao96} and Y. Caro \cite{Caro95},\cite{Caro96}.

\begin{thm}[Caro and Gao's $n+\mathsf{D}-1$ Theorem \cite{Gao96}, \cite{Caro95}, \cite{Caro96}]\label{thm:Gao96}
Let $G$ be a finite abelian group of order $n$.  The Erd\H{o}s-Ginzburg-Ziv constant of $G$ equals $n+\Dav{G}-1$.
\end{thm}

One of the aims of this paper is to explore to what extent this theorem may be generalized.  To do so, we first generalize the definition of these two constants, the former of which was previously done in work of the authors \cite{CaroGirardSchmitt21} and given again here.

Let $(G,+,\cdot)$ be a finite commutative ring. For any positive integer $m$ and any sequence $S=(g_1,\dots,g_\ell)$ over $G$, we set
$$e_m(S):=\displaystyle\sum_{1 \le i_1 < \dots < i_m \le \ell} \displaystyle\prod^m_{j=1} g_{i_j},$$ noting that operations are done coordinate-wise.

The introduction of this $m$-th degree symmetric polynomial expression given above allows for a generalization of both the Davenport constant and the Erd\H{o}s-Ginzburg-Ziv constant since both of these constants concern themselves with the vanishing of subsequences (with perhaps additional properties) on linear symmetric polynomial expressions.  We also mention that the question of the vanishing of a subsequence over certain symmetric polynomials has already appeared in \cite{AhmedBialostockiPhamLe19}, \cite{BialostockiLuong09}, and \cite{BialostockiLuong14}.

We denote by $\D{}{m}{G}$ the smallest positive integer $z$ such that every sequence $S$ over $G$ of length $|S| \ge z$ contains a subsequence $S'$ of length $|S'| \ge m$ for which $e_m(S')$ equals the zero-element in $G$.  Notice that when $m=1$ we recover the classical Davenport constant discussed above.  That is, $\D{}{1}{G}=\Dav{G}$ and in this case we prefer to use the notation $\Dav{G}$.  As a result and as $e_m(S)$ is a sum of products of degree $m$, we may consider $\D{}{m}{G}$ as the {\em $m^{th}$-degree Davenport constant}.  For results on $\D{}{m}{G}$, we refer the reader to earlier work done by the authors \cite{CaroGirardSchmitt21}.

For a finite commutative ring $G$, we denote by $\EGZ{t}{G}{m}$ the smallest positive integer $z$ such that every sequence $S$ over $G$ of length $|S| \ge z$ contains a subsequence $S'$ of length $t$ for which $e_m(S')$ evaluates to the zero-element in $G$.  If no such $z$ exists, we define $\EGZ{t}{G}{m} = \infty$.  

Of particular interest are cyclic groups $\Z_k$.  For integers $k$ and $m$, define $\s{k}{m} := \{ t: t\geq m ~{\text{and}}~ k\mid  {t \choose m} \}$.  For $k \geq 2, m \geq 1$ and $t \in \s{k}{m}$, we denote by $\EGZ{t}{\Z_k}{m}$ (or more simply $\EGZ{t}{k}{m}$) the smallest positive integer $z$ such that every sequence $S$ over $\Z_k$ of length $|S| \ge z$ contains a subsequence $S'$ of length $t$ for which $e_m(S')=0$.  The value of this function in the case $m=1$ was given by P.~Erd\H{o}s, A.~Ginzburg, A.~Ziv \cite{ErdosGinzburgZiv61} and is therefore called the {\em Erd\H{o}s-Ginzburg-Ziv constant}; they gave $\EGZ{k}{k}{1}=2k-1$.  As a result and as $e_m(S)$ is a sum of products of degree $m$, we may consider $\EGZ{t}{G}{m}$ as the {\em $m^{th}$-degree Erd\H{o}s-Ginzburg-Ziv constant of $G$}.

\begin{exa}
We give an example to show the usefulness of the condition $t \in \s{k}{m}$.  Let $m=2, t=8$ and $\Z_{10}$, i.e. $k=10$.  Let $S = (1^n)$ of length $n \geq t$.  For any length $8$ subsequence $S'$ of $S$, we have $e_2(S')={ 8 \choose 2} = 28$, which is not divisible by 10.
\end{exa}

Our results are as follows.  We begin in Section \ref{section:results} by providing a general lower bound on $\EGZ{t}{G}{m}$ for finite abelian groups $G$ in terms of $t, m$ and the $m^{th}$-degree Davenport constant.  In Subsection \ref{subsection:cyclicgroups} we focus on the case of when $G$ is a finite cyclic group, providing both lower and upper bounds for various instances of the parameters.  In Subsection \ref{subsection:Z2}, we give a precise determination of the function in the case that the group is $\Z_2$, showing that a generalization of Caro and Gao's $n+\mathsf{D}-1$ Theorem holds.  Such a generalization also holds in the case of $\Z_{p^s}$ when $t$ and $m$ are powers of the same prime as shown in Theorem \ref{thm:EGZprimepowers} and more generally for $p-$groups as shown as a consequence of Theorem \ref{thm:EGZforp-group}.  We frequently use polynomial methods or rely on results established using such methods.  We conclude our discussion in Section \ref{section:problems} with a conjecture and two problems.

\section{Results}\label{section:results}

First, we note an easy lower bound on $\D{\elem{m}{\bf x}}{m}{\Z_n}$.  Consider the sequence $\seq{1^t}$.  If $t=m$, then the only subsequence of length at least $m$ is the given sequence itself and $\elem{m}{\seq{1^t}}=1 \not \equiv 0 \pmod{n}$.  Further, suppose that for each $\ell$ with $t > \ell \geq m$ we have ${\ell \choose m} \not \equiv 0 \pmod{n}$.  Then there exists no subsequence of ${\seq{1^t}}$ of length at least $m$ which evaluates to zero modulo $n$.  Thus, we define $\Low{n}{m}$ to be the smallest integer $\ell \geq m+1$ such that ${\ell \choose m} \equiv 0 \pmod{n}$.  We have

\begin{equation}\label{eqn:L1sequence}
\D{\elem{m}{\bf x}}{m}{\Z_n} \geq \Low{n}{m}.
\end{equation}

Clearly, if $t \in \s{k}{m}$, then $t \geq max\{m+1, \Low{k}{m}\}$.  Also, note that for $k$ odd and $k \geq 3$ we have $\Low{k}{2}=k$.

As a first step towards exploring a Caro and Gao-type connection between the $m^{th}$-degree Davenport constant and the $m^{th}$-degree Erd\H{o}s-Ginzburg-Ziv constant, we provide a general lower bound on the latter.

\begin{thm}\label{thm:EGZversusDavenport}
Let $G$ be a finite abelian group (considered as a finite commutative ring).  Then $\EGZ{t}{G}{m} \geq t+\D{}{m}{G}-m$.
\end{thm}

\begin{proof}
If $\EGZ{t}{G}{m} = \infty$, we are done.  So, we may consider the cases where this parameter is finite.

Begin by noting that, by definition, we have $\D{}{m}{G} \geq m+1$.  Let $S^*$ be a sequence over $G$ of length $\D{}{m}{G} -1$ containing no subsequence $S^{**}$ for which $e_m(S^{**})=0$.  That is, $S^*$ is an extremal sequence for the $m^{th}$ degree Davenport constant.  Notice that $S^*$ does not contain the zero-element since otherwise any subsequence $S^{**}$ of length $m$ containing this zero-element would have $e_m(S^{**}) = 0$.  Let $S=\seq{0^{t-m},S^*}$, which has length $t+\D{}{m}{G}-m-1 \geq t$.  We will show that $S$ contains no subsequence $S'$ of length $t$ for which $e_m(S')=0$.  Any such sequence must have $t-j$ $0'$s and $j$ elements of $S^*$, where $m \leq j \leq \D{}{m}{G}-1$.  We then have that $e_m(S')=e_m(S^{**})$ for some $S^{**}$ a subsequence of $S^*$.  However, by construction, $e_m(S^{**}) \neq 0$, and so $e_m(S') \neq 0$.
\end{proof}

So, compare Theorem \ref{thm:EGZversusDavenport} to Theorem \ref{thm:Gao96}.  We will show that in particular cases equality holds in Theorem \ref{thm:EGZversusDavenport} but does not hold in general.  Note that Inequality \ref{eqn:L1sequence} and \ref{thm:EGZversusDavenport} immediately yield the following.

\begin{cor}\label{thm:EGZ-Lowlowerbound}
For $t \in \s{k}{m}$, we have $\EGZ{t}{k}{m} \geq t+\Low{k}{m}-m$.
\end{cor}


\subsection{Results for cyclic groups}\label{subsection:cyclicgroups}

\begin{prop}\label{prop:EGZ-generalupperbound}
For $t \in \s{k}{m}$, we have $\EGZ{t}{k}{m} \leq (k-1)(t-1)+t-m+1=k(t-1)-m+2$.
\end{prop}

\begin{proof}
Let $S$ be a sequence over $\Z_k$ of length $(k-1)(t-1)+t-m+1$.  If some non-zero element $g$ appears $t$ times, then there exists a subsequence $S'=(g^t)$ and we have $e_m(S') = g^m{t \choose m} \equiv 0 \pmod{k}$.  This implies that there are at most $(k-1)(t-1)$ non-zero elements in $S$ and at least $t-m+1$ elements which are $0$.  We may form the length-$t$ subsequence $S''=(0^{t-m+1}, g_1, g_2, \ldots , g_{m-1})$.  It is easy to see that $e_m(S'')=0$.
\end{proof}

Going further, while Proposition \ref{prop:EGZ-generalupperbound} gives that $\EGZ{3}{3}{2} \leq 6$, it is not hard to see that equality, in fact, holds by considering the length-$5$ sequence $S=(0,1^2,2^2)$ over $\Z_3$ and checking that for every subsequence $S'$ of $S$ with $|S'| =3$, one has $e_2(S') \not \equiv 0 \pmod 3$.

Now consider the following result given in \cite{CaroGirardSchmitt21}.

\begin{prop}\label{prop:Low}\cite{CaroGirardSchmitt21}
For a prime $p$ and integers $s$ and $u$, we have $\Low{p^s}{p^u} = p^{s+u}.$
\end{prop}

\begin{exa}
By Proposition \ref{prop:Low}, we have $\Low{5}{5}=25$.  Together with Theorem \ref{thm:EGZ-Lowlowerbound}, we have $\EGZ{25}{5}{5} \geq 45$.  Theorem \ref{thm:EGZprimepowers} given below will show this bound is sharp.
\end{exa}

\begin{thm}
Let $k$ be odd and let $r$ be an integer such that $r|k|r^2$.  Then $\D{}{2}{\Z_k} \leq k+r$. 
\end{thm}

\begin{proof}
Let $S=(g_1,\ldots, g_{k+r})$ be a sequence over $\Z_k$ of length $k+r$.  Consider the following sequence over $\Z_r \oplus \Z_k$: $([g_1, g_1^2], \ldots [g_{k+r},g_{k+r}^2])$.  As $r|k$ the group $\Z_r \oplus \Z_k$ is a rank-$2$ abelian group and so we may apply the Olson's Theorem (i.e. Theorem \ref{thm:Olson}) which says that $\Dav{\Z_r \oplus \Z_k}=k+r-1$.  That is, we have that there is a non-empty subset $J \subset [k+r]$ such that $\sum_{j \in J}g_j \equiv 0 \pmod r$ and $\sum_{j \in J}g_j^2 \equiv 0 \pmod k$.  We show that there exists such a $J$ such that $|J| \geq 2$. If not, then $|J|=1$ and so $g_1=0$.  Remove this element from the sequence $S$ to obtain $S\setminus g_1.$  Apply Olson's Theorem to this sequence and obtain a $J'$ with the same properties as $J$.  If $|J'|=1$ also, then $|J \cup J'| \geq 2$.  So we now may assume that we have a $J$ such that $|J| \geq 2$.  Observe that $(\sum_{j \in J}g_j)^2=\sum_{j \in J}g_j^2+2\sum_{i,j \in J, i\neq j}g_ig_j$.  Since $\sum_{j \in J}g_j \equiv 0 \pmod r$, it follows that $(\sum_{j \in J}g_j)^2 \equiv 0 \pmod {r^2}$.  As $k|r^2$, we have $(\sum_{j \in J}g_j)^2 \equiv 0 \pmod k$.  We now have that $(\sum_{j \in J}g_j)^2 \equiv 0 \pmod k$ {\em and} $\sum_{j \in J}g_j^2 \equiv 0 \pmod k$, implying that $2\sum_{i,j \in J, i\neq j}g_ig_j \equiv 0 \pmod k$.  However, as $k$ is odd it follows that $\sum_{i,j \in J, i\neq j}g_ig_j \equiv 0 \pmod k$.  Let $S'$ be the subsequence of $S$ as chosen by $J$, then $e_2(S') \equiv 0 \pmod k$.  

\end{proof}

\begin{thm}\label{thm:EGZupperbound}
Suppose that $k$ is odd. 
\begin{enumerate}
\item Let $\ell \geq 1$, and let $r$ be an integer such that $r|k|r^2$.  Then $\EGZ{\ell k}{k}{2} \leq (\ell +1)k+2r-3$.  In particular, if $k=r^2$, then $\EGZ{k}{k}{2} \leq 2r^2+2r-3$.
\item $\EGZ{k}{k}{2} \geq k+\D{}{2}{\Z_k}-2$.
\item In the case that $k=p$ is an odd prime: for $p \equiv 1 \pmod 4$ we have $\EGZ{p}{p}{2} \geq 2p-1$ and for $p \equiv 3 \pmod 4$ we have $\EGZ{p}{p}{2} \geq 2p$.
\end{enumerate}

\end{thm}

Before giving the proof of Theorem \ref{thm:EGZupperbound}, we need to recall a celebrated result that we use in the proof.  In 2007 C.~Reiher \cite{Reiher07} used the Chevalley-Warning Theorem and combinatorial arguments to establish a well-known conjecture of A.~Kemnitz; Reiher proved that the minimum number of points one needs to take from $\Z_n \oplus \Z_n$ so that there always exists $n$ of them summing to $(0,0)$ is $4n-3$.  This was generalized in \cite{GeroldingerHalter-Koch06} as follows.

\begin{thm}[C.~Reiher \cite{Reiher07}; A.~Geroldinger, F.~Halter-Koch \cite{GeroldingerHalter-Koch06}]\label{thm:Reiher}
Let $G=\Z_{n_1}\oplus \Z_{n_2}$ with $1 \leq n_1|n_2$.  Then $\EGZ{n_2}{\Z_{n_1} \oplus \Z_{n_2}}{1}=2n_1+2n_2-3.$
\end{thm}

For further remarks and results on $\EGZ{k}{\Z_k^d}{1}$, we point the reader to the work of N.~Alon and M.~Dubiner \cite{AlonDubiner93}.  We now proceed to the proof of Theorem \ref{thm:EGZupperbound}.

\begin{proof}
\begin{enumerate}

\item Let $S=(g_1,\ldots, g_{(\ell+1)k+2r-3})$ be a sequence over $\Z_k$ of length $(\ell+1)k+2r-3$.  Recall a basic algebraic fact: $2 \sum_{1 \leq i \neq j \leq d}g_ig_j = (g_1+\cdots +g_d)^2-(g_1^2+\cdots +g_d^2)$.  Consider the following sequence over $\Z_r \oplus \Z_k$: $$([g_1, g_1^2], \ldots [g_{2k+2r-3},g_{(\ell+1)k+2r-3}^2]).$$  As $r|k$ the group $\Z_r \oplus \Z_k$ is a rank-$2$ abelian group with exponent $k$.  For this group, by Theorem \ref{thm:Reiher} as $r|k$ we have $\EGZ{k}{\Z_r \oplus \Z_k}{1}=2k+2r-3$.  That is, we have that there exists $\ell$ disjoint non-empty subsets $J_1, \ldots, J_{\ell} \subset [(\ell+1)k+2r-3]$ with $|J_m|=k$ for $1 \leq m \leq \ell$ such that $\sum_{j \in J_m}g_j \equiv 0 \pmod r$ and $\sum_{j \in J_m}g_j^2 \equiv 0 \pmod k$.  Let $J:= \cup_{m=1}^{\ell} J_m$.  It follows that that $\sum_{j \in J}g_j \equiv 0 \pmod r$ and $\sum_{j \in J}g_j^2 \equiv 0 \pmod k$.  Since $\sum_{j \in J}g_j \equiv 0 \pmod r$, it follows that $(\sum_{j \in J}g_j)^2 \equiv 0 \pmod {r^2}$.  As $k|r^2$, we have $(\sum_{j \in J}g_j)^2 \equiv 0 \pmod k$.  We now have that $(\sum_{j \in J}g_j)^2 \equiv 0 \pmod k$ {\em and} $\sum_{j \in J}g_j^2 \equiv 0 \pmod k$, implying that $2\sum_{i,j \in J, i\neq j}g_ig_j \equiv 0 \pmod k$.  However, as $k$ is odd it follows that $\sum_{i,j \in J, i\neq j}g_ig_j \equiv 0 \pmod k$.  Let $S'$ be the subsequence of $S$ as chosen by $J$, then $e_2(S') \equiv 0 \pmod k$.   
 
 \item By Theorem \ref{thm:EGZversusDavenport}, if $t \in \s{k}{m}$, then $\EGZ{t}{k}{m} \geq \D{}{m}{\Z_k}+t-m$.  For $k \geq 3$, we have $k \in \s{k}{2}$ if and only if $k$ is odd.  Thus, $\EGZ{k}{k}{2} \geq \D{}{2}{\Z_k}+k-2$.

\item  By a result in \cite{CaroGirardSchmitt21} (see Section 5 of \cite{CaroGirardSchmitt21}), we have $\D{}{2}{\Z_p} \geq p+1$ for any prime and a stronger bound holds in the case that $p \equiv 3 \pmod 4$ of $\D{}{2}{\Z_p} \geq p+2$.  By this result and Part 2, the result follows. 
\end{enumerate}
\end{proof}

\begin{rem}
Note that the condition of $p$ being an odd prime give in Theorem \ref{thm:EGZupperbound}.3 is necessary as $S=\seq{1^t}$ shows that $\EGZ{2}{2}{2} = \infty.$
\end{rem}

\begin{thm}\label{thm:m=3case}

\leavevmode\\

\begin{enumerate}
\item  Let $k, m$ be positive integers such that $\gcd(k,m!)=1$.  Then $\EGZ{k}{k}{m} \leq \EGZ{k}{\Z_k^{\lceil \frac{m+1}{2} \rceil }}{1}$.  More strongly, if $\gcd(k,3)=1$, then $\EGZ{k}{k}{3} \leq 4k-3$.
\item Let $k$ be a positive integer such that $\gcd(k,3)=1$ (i.e. $k \equiv 1,2 \pmod 3$).  Then $\EGZ{k}{k}{3} \geq k+\D{}{3}{\Z_k}-3$.
\item Let $q$ be a prime power.  Then $\EGZ{q}{q}{3} \geq 2q-3$.
\end{enumerate}
\end{thm}

Before giving the proof of this theorem, we need to remind the reader of some classical results.

For $m \geq 0$, the {\it elementary symmetric polynomial of degree $m$} is the sum of all distinct products of $m$ distinct variables.  Thus, $\Elem{0}{n}=1, \Elem{1}{n}=x_1+\cdots +x_n, \Elem{2}{n}=\sum_{1 \leq i < j \leq n}x_ix_j$ and, so on, until,  $\Elem{n}{n} =x_1x_2\ldots x_n$.  The {\it $m$-th power sum polynomial} is $\Pow{m}{n} = \sum_{i=1}^{n} x_i^m$.

We now state a historical set of relations between the elementary symmetric polynomials and the power sum polynomials.  These 17th-century relations are independently due to Albert Girard and Isaac Newton and known as the Girard-Newton formulae (or sometimes Newton's identities); for more about these identities, see \cite{BeraMukherjee20}.

\begin{thm}[Girard-Newton formulae]
For all $n \geq 1$ and $1 \leq m \leq n$, we have
\begin{equation}\label{eqn:girard}
m\Elem{m}{n} = \sum_{i=1}^{m}(-1)^{i-1}\Elem{m-i}{n}\Pow{i}{n}.
\end{equation}

\end{thm}

We may rewrite Equation \ref{eqn:girard} in a manner that is independent of the number of variables, that is, we may rewrite Equation \ref{eqn:girard} in the ring of symmetric functions as

\begin{equation}\label{eqn:girardsimple}
me_m =\sum_{i=1}^m (-1)^{i-1} e_{m-i}p_i.
\end{equation}

One may use the Girard-Newton formulae to recursively express elementary symmetric polynomials in terms of power sums as follows.

\begin{equation}\label{eqn:elementaryinpower}
e_m =(-1)^m \sum \prod_{i=1}^{m} \frac{(-p_i)^{j_i}}{j_i!i^{j_i}},
\end{equation}
where the sum extends over all solutions to $j_1+2j_2+\cdots+mj_m=m$ such that $ j_1, \ldots, j_m \geq 0$.  For example, we have $e_1=p_1, ~e_2=\frac{1}{2}p_1^2-\frac{1}{2}p_2, ~e_3=\frac{1}{6}p_1^3-\frac{1}{2}p_1p_2+\frac{1}{3}p_3,~e_4=\frac{1}{24}p_1^4-\frac{1}{4}p_1^2p_2+\frac{1}{8}p_2^2+\frac{1}{3}p_1p_3-\frac{1}{4}p_4$. If we multiply both sides of Equation \ref{eqn:elementaryinpower} by $m!$, then we obtain on the right side integer coefficients.

Notice that for Equation \ref{eqn:elementaryinpower}, each term in the sum of the right side is a product that contains at most $m$ distinct power sum polynomials.  For a fixed $m$ we call a set $T$ of power sum polynomials a {\it dominating set for $e_m$} if each term in the sum contains at least one member of $T$. Let $\dom{m}$ denote the size of the smallest dominating set.  For $m=1$, the only dominating set is $\{p_1\}$, and so $t(1)=1$.  For $m=2$, the only dominating set is $\{p_1, p_2\}$, and so $t(2)=2$.  For $m=3$, any dominating set must contain both $p_1$ and $p_3$ and $\{p_1,p_3\}$ is a dominating set, and so $t(3)=2$.  More generally, the following was determined previously.

\begin{lem}[\cite{CaroGirardSchmitt21}]
We have $\dom{m}=\frac{m+2}{2}$ when $m$ is even, $\dom{m}= \frac{m+1}{2}$ when $m$ is odd.
\end{lem}

We are now able to give the proof of Theorem \ref{thm:m=3case}

\begin{proof}
\begin{enumerate}
\item Let $\GG = \EGZ{k}{\Z_k^{t(m)}}{1}$ and let $S=(g_1, \ldots, g_{\GG})$ be a sequence over $\Z_k$.  Let $D(m):=\{d_1,d_2, \ldots , d_{t(m)}\}$ be the exponents of a dominating set of size $t(m)$.  Consider the following sequence over $\Z_k^{t(m)}$: $$([g_1^{d_1}, g_1^{d_2},\ldots, g_1^{d_{t(m)}}], \ldots ,[g_{\GG}^{d_1},g_{\GG}^{d_2}, \ldots, g_{\GG}^{d_{t(m)}}]).$$  By the definition of $\EGZ{k}{\Z_k^d}{1}$, we have that there is a subset $J \subset [\GG]$ such that $\sum_{j \in J}[g_j^{d_1}, g_j^{d_2}, \ldots, g_j^{d_{t(m)}}] = (0,\ldots, 0)$.  Thus, $\sum_{j \in J}g_j^{d_1} \equiv 0 \pmod k$, $\sum_{j \in J}g_j^{d_2} \equiv 0 \pmod k$, and all the way to $\sum_{j \in J}g_j^{d_{t(m)}} \equiv 0 \pmod k$.  Let $S'$ be the sequence selected by $J$. We use the fact that the Girard-Newton formulae allow us to express $m!e_m$ as a sum whose terms consist of power sum polynomials.  In this sum at least one factor in each term is equal to $0 \pmod k$ since through $D(m)$ we have created a dominating set.  Thus, the sum is $0 \pmod k$.  It follows that $e_m(S') \equiv 0 \pmod k$.

The ``in particular" statement follows by the application of Theorem \ref{thm:Reiher}.  The weaker $gcd$ condition follows from the fact that by Equation \ref{eqn:girardsimple} we have the expression $3e_3=e_2p_1-e_1p_2+p_3$ and as $e_1=p_1$ we may write $3e_3=e_2p_1-p_1p_2+p_3$.  The set $\{p_1,p_3\}$ is a dominating set for this expression.

\item By Theorem \ref{thm:EGZversusDavenport} if $t \in \s{k}{m}$, then $\EGZ{t}{k}{m} \geq \D{}{m}{\Z_k}+t-m$.  For $k \geq 4, k \in \s{k}{3}$ if and only if $k \not \equiv 0 \pmod 3$.  Thus, $\EGZ{k}{k}{3} \geq k+\D{}{3}{\Z_k}-3$.


\item The sequence $S=(1^{q-1})$ shows that $\D{}{3}{\Z_q} \geq q$.  Thus, $\EGZ{q}{q}{3} \geq q+\D{}{3}{\Z_q}-3=2q-3$.

\end{enumerate}

\end{proof}




\begin{prob}\label{prob:bounds?}
Determine a lower bound for $\EGZ{k}{k}{m}$ for $k \in \s{k}{m}$ and an upper bound for $\EGZ{k}{k}{m}$ for all $k,m$.
\end{prob}



\subsubsection{Exact determination for $\Z_2$}\label{subsection:Z2}

In order to prove a generalization of the Caro-Gao Theorem for the cyclic group of order $2$, we need some preparatory lemmas.  To prove these lemmas, we provide some necessary background.

Let $p$ be a prime number and $n > 1$ an integer.  The {\it $p$-adic valuation of $n$}, denoted $\nu_p(n)$, is the exponent of $p$ in the canonical decomposition in prime numbers of $n$ (and if $p$ does not divide $n$, then $\nu_p(n)=0$).  The base-$p$ expansion of $n$ is written as such, $n=a_kp^k+a_{k-1}p^{k-1}+\cdots +a_1p+a_0$.  




\begin{thm}[E.~Kummer, 1852 \cite{Kummer1852}]
The $p-$adic valuation of the binomial coefficient ${n \choose m}$ is equal to the number of `carry-overs' when performing the addition in base $p$ of $n-m$ and $m$. 
\end{thm}




Recall Inequality \ref{eqn:L1sequence}: $\D{}{m}{\Z_n} \geq \Low{n}{m}$ and the discussion that precedes it. In particular, we have $\D{}{m}{\Z_2} \geq \Low{2}{m}$.  Notice that any sequence $S$ containing the element $0$ has a subsequence $S$ for which $e_m(S')=0$.  Any sequence $S$ not containing $0$ and of length $\Low{2}{m}$, we have $e_m(S) \equiv 0 \pmod 2$.

\begin{lem}\label{lem:mplus2^nu}
$\D{}{m}{\Z_2} = m+2^{\nu_2(m)}.$
\end{lem}

\begin{proof}
Let $m$ be a positive integer with $2-$adic valuation $\nu_2(m)$.  

Notice that any sequence $S$ containing the element $0$ has a subsequence $S$ for which $e_m(S')=0$.  We will show that any sequence $S$ not containing $0$ and of length $m+2^{\nu_2(m)})$, that is $S=(1^{m+2^{\nu_2(m)}})$ we have $e_m(S) = {m+2^{\nu_2(m)} \choose m} \equiv 0 \pmod 2$ and sequences of shorter length do not have this property.

We begin by showing that for $m+1 \leq j < m+2^{\nu_2(m)}$ we have ${j \choose m} \not \equiv 0 \pmod 2$.  


We will use Kummer's Theorem to compute $2$-adic valuation of this binomial coefficient ${j \choose m}$: it is equal to the number of `carry-overs' when performing the addition in base $2$ of $j-m$ and $m$.  Begin by noticing that base-2 expansion of $m$ has $0'$s to the right of position $\nu_{2}(m)$ (and a $1$ in position $\nu_2(m)$).  That is, the base-$2$ expansion of $m$ ends with $0\cdot2^{\nu_2(m)-1}+\cdots + 0\cdot2^0$.  As $m+1 \leq j < m+2^{\nu_2(m)}$, we have $1 \leq j-m < 2^{\nu_2(m)}$.  Thus, the base-$2$ expansion does not have a $1$ to the left of position $\nu_2(m)$.  Thus, when we add $m$ and $j-m$ in base-$2$, there are no `carry-overs'.  By Kummer's Theorem, the $2$-adic valuation of  ${j \choose m}$ is $0$.  That is, $2$ does not divide ${j \choose m}$ for $m+1 \leq j < m+2^{\nu_2(m)}$.

Now consider $j=m+2^{\nu_2(m)}$.  In this case, $j-m=2^{\nu_2(m)}$.  Thus, the base-$2$ expansion of $j-m$ is $1\cdot2^{\nu_2(m)}$.  As $m$ has a $1$ in position $\nu_2(m)$, when we add $m$ and $j-m$ in base-$2$ there is at least one `carry-over'.  Thus, some positive power of $2$ divides ${m+2^{\nu_2(m)} \choose m}$.
\end{proof}

\begin{lem}\label{lem:intervalhaswhatweneed}
Suppose that $i \geq m+2^{\nu_2(m)}$.  Then for some $j$ in the interval $[i-2^{\nu_2(m)}, \ldots, i]$, we have ${j \choose m} \equiv 0 \pmod 2$.  
\end{lem}

\begin{proof}
As the interval has length $2^{\nu_2(m)}+1$, there exists some integer $j$ with a $0$ in position $\nu_2(m)$ of its base-$2$ expansion and $j-m \geq 2^{\nu_2(m)}$.  As $m$ has a $1$ in position $\nu_2(m)$, $j-m$ has $1$ in position $\nu_2(m)$ of the base-$2$ expansion.  Thus, when we add $m$ and $j-m$ in base-$2$ there is at least one `carry-over'.  Thus, by Kummer's Theorem, some positive power of $2$ divides ${j \choose m}$.
\end{proof}

The next Theorem is showing that a generalized Caro-Gao Theorem holds in a particular case.

\begin{thm}\label{thm:generalizedGaok=2}
For $t \in \s{2}{m}$, we have $\EGZ{t}{2}{m} = t+2^{\nu_2(m)}=t+\D{}{m}{\Z_2}-m.$

\end{thm}

\begin{proof}
By Lemma \ref{lem:mplus2^nu}, we only need show $\EGZ{t}{2}{m} = t+2^{\nu_2(m)}$.

By Theorem \ref{thm:EGZversusDavenport}, we have $\EGZ{t}{2}{m} \geq t+2^{\nu_2(m)}.$

We now prove the upper bound.  Let $S=(0^{t+2^{\nu_2(m)}-i}, 1^i)$ be a sequence of length $t+2^{\nu_2(m)}$, where $0 \leq i \leq t+2^{\nu_2(m)}$.  We consider several cases, in each case showing that there exists a subsequence $S'$ of length $t$ such that $e_m(S') \equiv 0 \pmod 2$.

\begin{enumerate}
\item $i \geq t \geq m+1$. \\
There exists the subsequence $S'=\seq{1^t}$.  As $t \in \s{2}{m}$, we have $e_m(S')={t \choose 2} \equiv 0 \pmod 2$.

\item $i \leq m-1$.\\
Then the number of $0'$s that $S$ contains is at least $t+2^{\nu_2(m)}-i \geq t+1-i > t-i$.  Let $S'=\seq{0^{t-i},1^i}$.  As each summand in $e_m(S')$ equals $0$, we have $e_m(S') \equiv 0 \pmod 2$.

\item $m \leq i \leq m+2^{\nu_2(m)}-1$.\\
Then the number of $0'$s that $S$ contains is at least $t+2^{\nu_2(m)}-i \geq t-m+1$.  Let $S'=\seq{0^{t-m+1},1^{m-1}}$.  As each summand in $e_m(S')$ equals $0$, we have $e_m(S') \equiv 0 \pmod 2$.

\item $m+2^{\nu_2(m)} \leq i \leq t-1$.\\
Then the number of $0'$s that $S$ contains is at least $t+2^{\nu_2(m)}-i \geq 2 ^{\nu_2(m)}+1$.  By Lemma \ref{lem:intervalhaswhatweneed}, there exists a $j \in [i-2^{\nu_2(m)},\ldots, i]$ such that ${i-j \choose m} \equiv 0 \pmod 2$.  Let $S'=\seq{0^{t-(i-j)},1^{(i-j)}}$.  Each summand in $e_m(S')$ is either $0$ or $1$, and the number of the latter is ${i-j \choose m}$.  Thus, $e_m(S') = {i-j \choose m} \equiv 0 \pmod 2$.
\end{enumerate}
\end{proof}

\subsection{Prime power parameters yield a Caro-Gao-type theorem}\label{subsection:primepower}

In this subsection, we investigate the $\EGZ{t}{k}{m}$ when the parameters are restricted to being powers of the same prime.  Central to establishing our results here and later is the use of a tool from the polynomial method tool-kit, as follows.

\begin{thm}[U.~Schauz \cite{Schauz08}, D.~Brink \cite{Brink11}]
\label{RVSCHANUEL}
\label{thm:BIGBRINKTHM}
Let $P_1(t_1,\ldots,t_n),\ldots,P_r(t_1,\ldots,t_n) \in \Z[t_1,\ldots,t_n]$ be 
polynomials, let $p$ be a prime, let $v_1,\ldots,v_r \in \Z^+$, and 
let $A_1,\ldots,A_n$ be nonempty subsets of $\Z$ such that for each $i$, 
the elements of $A_i$ are pairwise incongruent modulo $p$, and put 
$A = \prod_{i=1}^n A_i$. 
 Let 
\[ Z_A = \{x \in A \mid P_j(x) \equiv 0 \pmod{p^{v_j}} \ \forall 1 \leq j \leq r \}, \ \zz_A = \# Z_A. \]
a) If $\sum_{j=1}^r (p^{v_j}-1)\deg(P_j) < \sum_{i=1}^n \left( \#A_i - 1 \right)$, then $\zz_A \neq 1$. \\
b) (\text{Boolean Case}) If $A = \{0,1\}^n$ and $\sum_{j=1}^r (p^{v_j}-1)\deg(P_j) < n$, then $\zz_A \neq 1$. 
\end{thm}

When we apply Theorem \ref{thm:BIGBRINKTHM}, most notably the Boolean Case, we will use a system of polynomials to encode the combinatorial problem in the zero set of this system.  When applied, the Boolean Case will guarantee the existence of a non-zero boolean vector in the zero set.  If the $j^{th}$-entry of this vector is $1$, then this will correspond to selecting the $j^{th}$-entry of a given sequence $S$ whereas $0$ corresponds to not selecting this entry.  In the proofs that make use of this theorem, we will write it so that the first polynomial (or set of polynomials) that we give will ensure that we pick out a subsequence that sums to the zero-element and the second polynomial that we give will ensure that the subsequence that we pick out is of the desired length.  We note that Theorem \ref{thm:BIGBRINKTHM} has been generalized (see \cite{ClarkForrowSchmitt17}), though the statement given here is sufficient for our purposes.

To warm the reader to the employment of this method, we begin with an example.

\begin{exa}
We compute $\EGZ{16}{8}{2}$.  First, note that $16 \in \s{8}{2}$. Let $g_1, \ldots ,g_{30}$ be integers.

Let $P = \displaystyle\sum_{1 \le i < \dots < j \le 30}  g_ig_jx_ix_j$ and $Q = \displaystyle\sum_{1 \le i \le 30} x_i$. We seek a particular type of member of the set of shared zeros of $P \equiv 0 \pmod{2^3}, Q \equiv 0 \pmod{2^4}$.  We use the Boolean Case of Theorem \ref{thm:BIGBRINKTHM}.  First note that the zero-vector is a shared zero of this polynomial system.  Note that the hypothesis of Theorem \ref{thm:BIGBRINKTHM} is satisfied, that is, we have $(2^3-1)deg(P)+(2^4-1)deg(Q)=7*2+15*1=29 < 30$.  Thus, there exists a shared zero other than the zero-vector.  This boolean vector of length $30$ must have precisely $16$ $1'$s in it as guaranteed by $Q \equiv 0 \pmod {2^4}$.  These $1'$s select a subsequence $S'$ of the above list of integers of length 16 such that $e_2(S') \equiv 0 \pmod{8}$.  Thus, $\EGZ{16}{8}{2} \leq 30$.

We now show that $\EGZ{16}{8}{2} >29$.  Consider the following sequence $S=(0^{14},1^{15})$.  We show there is no subsequence $S'$ of length 16 such that $e_2(S') \equiv 0\pmod{8}$.  Let $x$ count the number of $1'$s in any subsequence $S'$.  Then $e_2(S') \equiv {x \choose 2} \pmod{8}$.  However, as we have $2 \leq x \leq 15$, ${x \choose 2} \not \equiv 0 \pmod{8}$.

Thus, $\EGZ{16}{8}{2}=30$.
\end{exa}

Notice that Part 3 of the following theorem is a Caro-Gao-type statement.

\begin{thm}\label{thm:EGZprimepowers}
\begin{enumerate}
\item Let $r$ and $s$ be positive integers with $r  \geq s$, $p$ a prime, $m \geq 1$ and $p^{r} > mp^{s}-m$. We have $\EGZ{p^{r}}{p^s}{m} \leq p^{r}+mp^s-m.$
\item Let $t \in \s{p^s}{p^u}$.  Then $\EGZ{t}{p^s}{p^u} \geq t+p^{s+u}-p^u$.  Furthermore, if $t=p^r$ where $r > u$, then $\EGZ{p^r}{p^s}{p^u} \geq p^r+p^{s+u}-p^u$.
\item Let $r, s, u$ be positive integers with $r \geq s+u$.  Then $\EGZ{p^r}{p^s}{p^u}=p^r+p^{s+u}-p^u$.

\end{enumerate}
\end{thm}
\begin{proof}
\begin{enumerate}
\item Let $S=(g_1, \ldots, g_{p^{r}+mp^s-m})$ be a sequence over $\Z_{p^s}$.
Let $$P = \displaystyle\sum_{1 \le i_1 < \dots < i_m \le p^{r}+mp^s-m}  g_{i_1}\cdots g_{i_m}x_{i_1}\cdots x_{i_m},$$  $$Q = \displaystyle\sum_{1 \le i \le p^{r}+mp^s-m} x_i.$$ We seek a particular type of member of the set of shared zeros of $P \equiv 0 \pmod{p^s}, Q \equiv 0 \pmod{p^{r}}$.  We use the Boolean Case of Theorem \ref{thm:BIGBRINKTHM}.

First note that the zero-vector is a shared zero of this polynomial system.  Note that the hypothesis of Theorem \ref{thm:BIGBRINKTHM} is satisfied, that is, we have $(p^s – 1)deg(P) +(p^{r} – 1)deg(Q) =  p^{r} + mp^s  - (m+1)   <  p^{r}+mp^s-m$.  Thus, there exists a shared zero other than the zero-vector.  This boolean vector of length $p^{r}+mp^s-m$ must have precisely $p^{r}$ $1'$s in it as  $Q \equiv 0 \pmod {p^{r}}$ and by hypothesis $p^{r}+mp^s-m < 2p^{r}$.  These $1'$s select a subsequence $S'$ of the sequence $S$ such that $e_m(S')\equiv 0 \pmod{p^s}$.  Thus, $\EGZ{p^{r}}{p^s}{m} \leq p^{r}+mp^s-m$.

\item By Theorem \ref{thm:EGZversusDavenport}, if $t \in \s{k}{m}$, then $\EGZ{t}{k}{m} \geq t+\D{}{m}{\Z_k}-m$.  By a result in \cite{CaroGirardSchmitt21} (see Theorem 2.7 of that paper), we know that $\D{}{p^u}{\Z_{p^s}} = p^{s+u}$.  Hence $\EGZ{t}{p^s}{p^u} \geq t+p^{s+u}-p^u$.  In the case that $t=p^r$ and as $t \geq m+1$, we then know that $r >u$.  As a result, $\EGZ{p^r}{p^s}{p^u} \geq p^r+p^{s+u}-p^u$

\item With $m=p^u$ and $r \geq s+u$, we have $p^r \geq p^{s+u} > p^u(p^{s-1})$.  By Part 1, we have $\EGZ{p^{r}}{p^s}{p^u} \leq p^{r}+p^up^s-p^u.$  By Part 2, we infer that equality holds.

\end{enumerate}
\end{proof}

\subsubsection{Results for $p$-groups}

\begin{thm}\label{thm:EGZforp-group}

Let $p$ be a prime.  Let $G=\pgroup{r}$ be a $p-$group of rank $r$.  Let $h=\sum_{i=1}^r {\alpha_i}$.
\begin{enumerate}
\item Suppose that $p^h> \sum_{j=1}^rm(p^{\alpha_j}-1)$.  Then $\EGZ{p^h}{\pgroup{r}}{m} \leq   p^h+m\sum_{j=1}^r(p^{\alpha_j}-1)$.
\item $\EGZ{t}{\pgroup{r}}{p^s} \geq t+p^s\sum_{j=1}^{r}(p^{\alpha_j}-1).$
\item Suppose that $p^h> \sum_{j=1}^rp^s(p^{\alpha_j}-1)$.  Then $\EGZ{p^h}{\pgroup{r}}{p^s} =   p^h+p^s\sum_{j=1}^r(p^{\alpha_j}-1)$.
\end{enumerate}
\end{thm}

\begin{proof}
\begin{enumerate}
\item Let $\omega = p^h+\sum_{j=1}^rm(p^{\alpha_j}-1)$.
Let $S=(g_1, \ldots, g_{\omega})$ be a sequence over $\pgroup{r}$, where $g_i=(a^{(1)}_i,\ldots,a^{(r)}_i)$.
For each $1 \leq j \leq r$, 

$$P_j = \displaystyle\sum_{1 \le i_1 < \cdots < i_m \le \omega}  a^{(j)}_{i_1}\cdots a^{(j)}_{i_m}x_{i_1} \cdots x_{i_m},$$  
$$Q = \displaystyle\sum_{1 \le i \le \omega} x_i.$$

We seek a particular type of member of the set of shared zeros of $P_j \equiv 0 \pmod{p^{\alpha_j}}, Q \equiv 0 \pmod{p^{h}}$ for $1 \leq j \leq r$.  We use the Boolean Case of Theorem \ref{thm:BIGBRINKTHM}. First note that the zero-vector is a shared zero of this polynomial system.  Note that the hypothesis of Theorem \ref{thm:BIGBRINKTHM} is satisfied, that is, we have $\sum_{j=1}^{r}(p^{\alpha_j} – 1)deg(P_j) +(p^{\sum_{j=1}^r \alpha_j}-1)deg(Q)=\sum_{j=1}^{r}(p^{\alpha_j} – 1)m +(p^h-1) < \omega$.  Thus, there exists a shared zero other than the zero-vector.  This Boolean vector of length $\omega$ must have precisely $p^h$ $1'$s in it as  $Q \equiv 0 \pmod {p^h}$ and by hypothesis $p^h> \sum_{j=1}^rm(p^{\alpha_j}-1)$ .  These $1'$s select a subsequence $S'$ of length $p^{h}$ such that $e_m(S')$ evaluates to the zero-element in the ring.  Thus, $\EGZ{p^h}{\mathbb{Z}_{p^{\alpha_1}}\oplus \ldots \oplus \mathbb{Z}_{p^{\alpha_r}}}{m} \leq \omega$.

\item By Theorem \ref{thm:EGZversusDavenport}, we have $\EGZ{t}{G}{m} \geq t+\D{}{m}{G}-m$.  By a result in \cite{CaroGirardSchmitt21} (see Theorem 3.7 of \cite{CaroGirardSchmitt21}), we know that $\D{\elem{p^s}{\bf x}}{p^s}{\pgroup{r}} = p^s\left(\sum_{i=1}^r(p^{\alpha_i}-1)+1\right)=p^s\Dav{\pgroup{r}}$.  Hence, 

\begin{eqnarray*}
\EGZ{t}{\pgroup{r}}{p^s} & \geq & t+p^s\left(\sum_{i=1}^r(p^{\alpha_i}-1)+1\right)-p^s \\
& = &t+p^s\sum_{i=1}^r(p^{\alpha_i}-1).
\end{eqnarray*}

\item Part 1 with $m=p^s$ and Part 2 together imply the result.
\end{enumerate}

\end{proof}

An immediate consequence is another Caro-Gao-type statement.

\begin{cor}\label{cor:EGZforp-group}
Let $p$ be a prime.  Let $G=\pgroup{r}$ be a $p-$group of rank $r$.  Let $h=\sum_{i=1}^r {\alpha_i}$.  Suppose that $p^h> \sum_{j=1}^rp^s(p^{\alpha_j}-1)$.  Then $\EGZ{p^h}{\pgroup{r}}{p^s} =   p^h+\D{}{p^s}{\pgroup{r}}-p^s$.
\end{cor}

\begin{proof}
It was shown in \cite{CaroGirardSchmitt21} that $\D{}{p^s}{\pgroup{r}}=p^s(\sum_{j=1}^r(p^{\alpha_j}-1)+1)$.  Thus, by Theorem \ref{thm:EGZforp-group}.3, the result follows immediately.
\end{proof}

\begin{thm}
Let $p$ be a prime such that $p > m$.  Let $G=\mathbb{Z}_{p^{\alpha_1}}\oplus \ldots \oplus \mathbb{Z}_{p^{\alpha_r}}$ be a $p-$group of rank $r$.  Let $h=\sum_{i=1}^r {\alpha_i}$ and suppose that $p^h> (\lfloor \frac{m}{2} \rfloor +1)(\sum_{i=1}^rp^{\alpha_i}-1)$.  Then $\EGZ{p^h}{\mathbb{Z}_{p^{\alpha_1}}\oplus \ldots \oplus \mathbb{Z}_{p^{\alpha_r}}}{m} \leq   p^h+(\lfloor \frac{m}{2} \rfloor +1)(\sum_{i=1}^rp^{\alpha_i}-1)$.
\end{thm}

\begin{proof}
Let $\omega = p^h+(\lfloor \frac{m}{2} \rfloor +1)(\sum_{i=1}^rp^{\alpha_i}-1)$.
Let $S=(g_1, \ldots, g_{\omega})$ be a sequence over $\mathbb{Z}_{p^{\alpha_1}}\oplus \ldots \oplus \mathbb{Z}_{p^{\alpha_r}}$, where $g_i=(a^{(1)}_i,\ldots,a^{(r)}_i)$.

For each coordinate $j=1,\ldots, r$, we define a set of $\lfloor \frac{m}{2} \rfloor +1$ linear polynomials as follows.  Let $u \in \{1,\ldots ,\lfloor \frac{m}{2} \rfloor, m\}$.  Then define

$$P_{j,u} =  \sum_{i=1}^{\omega} (a^{(j)}_i)^ux_i.$$

Define another linear polynomial $$Q = \sum_{i=1}^{\omega}x_i.$$

We seek a particular type of member of the set of shared zeros of $P_{j,u} \equiv 0 \pmod{p^{\alpha_j}}$ for $j \in \{1, \ldots , \lfloor \frac{m}{2} \rfloor, m\}$ and $Q \equiv 0 \pmod{p^{h}}$.  We use the Boolean Case of Theorem \ref{thm:BIGBRINKTHM}.

First note that the zero-vector is a shared zero of this polynomial system.  Note that the hypothesis of Theorem \ref{thm:BIGBRINKTHM} is satisfied, that is, we have 
\begin{eqnarray*}
\omega &  > & (p^h -1)deg(Q) + \sum_{j=1}^r \sum_{u \in \{1,\ldots ,\lfloor \frac{m}{2} \rfloor, m\}} (p^{\alpha_j}-1)deg(P_{j,u}) \\
 & = & p^h-1+(\lfloor \frac{m}{2} \rfloor +1)(\sum_{i=1}^rp^{\alpha_j}-1).
 \end{eqnarray*}
 
Thus, there exists a shared zero other than the zero-vector.  This boolean vector of length $\omega$ must have precisely $p^h$ $1'$s in it as $Q \equiv 0 \pmod{p^h}$ and by the hypothesis $2p^h > \omega$.
 
 We now must show that these $1'$s select a length-$p^h$ subsequence $S'$ of the above list of such that $e_m(S')$ equals the zero-element in $\pgroup{r}$. 
 
We point out that $\{p_1, \ldots , p_{\lfloor \frac{m}{2} \rfloor}, p_m\}$  is a minimum dominating set for the elementary symmetric polynomial $e_m$ (see Section 4 of \cite{CaroGirardSchmitt21}).  By the Newton-Girard formulae, $m!e_m$ may be written as a sum of products, where each product has an integer coefficient and each product contains at least one of the elements from the above dominating set.  Now as each of these polynomials in the dominating set evaluate to $0$ on $S'$, so does $m!e_m$.  The hypothesis that $p >m$, implies that $\gcd(p, m!)=1$.  This now implies that $e_m(S')$ equals the zero-element.

\end{proof}

\section{Problems and conjectures}\label{section:problems}

Recall,\\
 {\bf Problem \ref{prob:bounds?}.}
Determine a lower bound for $\EGZ{k}{k}{m}$ for $k \in \s{k}{m}$ and an upper bound for $\EGZ{k}{k}{m}$ for all $k,m$.\\

Motivated by Theorem \ref{thm:generalizedGaok=2} and computations provided to us by Benjamin Girard \cite{Girard22}, we give the following.

\begin{conj}\label{conj:Gaoconjectureqqcase}
$\EGZ{t}{q}{q} = t+q^2-q=t+\D{}{q}{\Z_q}-q.$
\end{conj}

It is certainly the case that something more nuanced is true in the case that $k$ and $m$ are not both equal to a prime power $q$.  This is evidence by the following computations provided to us by Benjamin Girard \cite{Girard22}:

$$\EGZ{9}{9}{2}=17 >9+\D{}{2}{\Z_9}-2=9+9-2=16,$$ and

$$\EGZ{10}{6}{6}=19 > 10+\D{}{6}{\Z_6}-6=10+13-6=17.$$



\begin{prob}
Determine an upper bound for the $m^{th}$-degree Erd\H{o}s-Ginzburg-Ziv constant for general abelian groups.
\end{prob}

{\bf Acknowledgments:}  We give thanks to BIRS-CMO 2019 and Casa Matem\'atica Oaxaca, Mexico for supporting and hosting the event Zero-Sum Ramsey Theory: Graphs, Sequences and More 19w5132.  We are grateful for helpful comments from Qinghai Zhao.

\end{document}